\documentclass[12pt,thmsa, reqno]{amsart}

\usepackage{mathrsfs}
\usepackage{amsfonts}
\usepackage{}
\usepackage{amsmath}

\usepackage[dvips]{graphics}
\input{amssym.def}
\input{amssym.tex}

\def\gnk{G_{n,k}}
\def\agnk{ G(n,k)}
\def\gnkp{G_{n,k'}}
\def\agnkp{ G(n,k')}
\def\ia{I_+^\alpha}

\def\rk{\bbr^k}
\def\rkp{\bbr^{k'}}
\def\rkk{\bbr^{k'-k}}

\def\rnkp{\bbr^{n-k'}}

\def\iagr{\intl_{\agnk}}
\def\iagrp{\intl_{\agnkp}}

\def\Cal{\mathcal}

\def\P{{\Cal P}}
\def\D{{\Cal D}}

\def\F{{\Cal F}}

\def\gnk{G_{n,k}}
\def\gnkp{G_{n,k'}}

\def\ia{I_+^\alpha}

\def\pk{\P_k}

\def\bbr{{\Bbb R}}

\def\bbn{{\Bbb N}}

\def\bbc{{\Bbb C}}

\def\dist{{\hbox{\rm dist}}}

\def\const{{\hbox{\rm const}}}

\def\Cos{{\hbox{\rm Cos}}}
\def\det{{\hbox{\rm det}}}

\def\Pr{{\hbox{\rm Pr}}}

\def\gnk{G_{n,k}}
\def\gnkp{G_{n,k'}}

\def\ia{I_+^\alpha}

\def\pk{\P_k}

\def\rn{\bbr^n}

\def\rk{\bbr^k}

\def\part{\partial}
\def\intl{\int\limits}
\def\b{\beta}

\def\Gam{\Gamma}
\def\Om{\Omega}
\def\a{\alpha}
\def\om{\omega}

\def\del{\delta}
\def\vp{\varphi}

\def\vk{\varkappa}
\def\g{\gamma}
\def\gam{\gamma}
\def\Lam{\Lambda}
\def\sig{\sigma}
\def\lam{\lambda}
\def\z{\zeta}

\def\t{\tau}
\def\th{\theta}

\font\frak=eufm10

\def\fr#1{\hbox{\frak #1}}

\newtheorem{theorem}{Theorem}[section]
\newtheorem{lemma}[theorem]{Lemma}

\theoremstyle{definition}
\newtheorem{definition}[theorem]{Definition}

\theoremstyle{remark}
\newtheorem{remark}[theorem]{Remark}

\theoremstyle{corollary}

\newtheorem{proposition}[theorem]{Proposition}

\numberwithin{equation}{section}

\newcommand{\be}{\begin{equation}}
\newcommand{\ee}{\end{equation}}

\newcommand{\bea}{\begin{eqnarray}}
\newcommand{\eea}{\end{eqnarray}}
\newcommand{\Bea}{\begin{eqnarray*}}
\newcommand{\Eea}{\end{eqnarray*}}

\def\sideremark#1{\ifvmode\leavevmode\fi\vadjust{\vbox to0pt{\vss
 \hbox to 0pt{\hskip\hsize\hskip1em
\vbox{\hsize2cm\tiny\raggedright\pretolerance10000
 \noindent #1\hfill}\hss}\vbox to8pt{\vfil}\vss}}}%

\begin{document}

\title [Radon Transforms on Affine  Grassmannians]{ New Inversion Formulas for Radon Transforms on Affine  Grassmannians}

\author{Boris Rubin and Yingzhan Wang*}
\address{
Department of Mathematics, Louisiana State University, Baton Rouge,
LA, 70803 USA}
\email{borisr@math.lsu.edu}
\address{School of Mathematics and Information Science, Guangzhou University, Guangzhou 510006, China;}
\email{wyzde@gzhu.edu.cn}
\thanks{*Y. Wang was supported by NNSF of China (Grant Nos. 10471040)}





\begin{abstract}
We obtain new inversion formulas for the Radon transform and the corresponding dual transform acting on affine Grassmann manifolds of planes in $\rn$. The consideration is performed in full generality on continuous functions and functions belonging to  $L^p$ spaces.

\end{abstract}

\maketitle

\section{Introduction}

This article is  a continuation and generalization of our previous work  \cite {RW} devoted to integral geometry on  line bundles.  Let $\agnk$ and $\agnkp$ be a pair  of the affine Grassmann
manifolds of   $k$-dimensional  and $k'$-dimensional non-oriented  planes  in
$\rn$, respectively. We suppose that $0 < k <k'<n$. The excluded case $k=0$
 formally corresponds to points in $\rn$. Given sufficiently good functions
$f$ on $\agnk$ and $\vp$ on $\agnkp$, we consider the
following integral transforms
\be\label{op1} (Rf)(\z)  = \intl_{\t \subset \zeta}
f(\t)\, d_\z \t, \qquad (R^*\vp)(\t) = \intl_{\z \supset \t} \vp(\z)\, d_\t \z, \ee
the
integration being performed with respect to the corresponding canonical
measures. The first integral is called {\it the Radon transform}
of $f$ and denotes integration over all $k$-planes $\t$ in the
$k'$-plane $\z$. The second one is called {\it the dual Radon
transform} of $\vp$ and  integrates over all $k'$-planes
$\z$ containing  the $k$-plane $\t$.

Our goal is to find explicit inversion formulas for these transforms in possibly wide classes of functions. The  problem is not new.
Similar problems for compact Grassmann manifolds $\gnk, \,\gnkp$
 of  $k$-dimensional and $k'$-dimensional linear subspaces  of $\rn$ were studied by many authors, including  I.M. Gelfand,
 M.I. Graev, Z.Ya. Shapiro,  R.  Ro\c{s}u,  E.E. Petrov, E.L. Grinberg, F. Gonzalez, T.
 Kakehi, G. Zhang: see   \cite{GR04, Zha1}, and references therein.

 The noncompact case   is essentially more complicated.  Different approaches to the study of  operators (\ref{op1}) are known.  M.I. Graev \cite{Gra} parametrized planes in $\rn$ by
 matrices and  obtained an inversion
 formula for $R f$ in the so-called local case (when  $k'-k$ is even) by making use of differential forms and kappa operators.  Another inversion formula  for $R f$ if
 $k'-k$ is even was suggested  by F. Gonzalez and T. Kakehi  \cite{GK03}
 who used the corresponding Lie algebra language and the Fourier transform techniques. In both publications only smooth rapidly decreasing functions $f$ were considered.
 One should also mention the paper by Strichartz \cite{Str86} who
 developed  $L^2$ harmonic analysis on Grassmannian bundles.

A completely different approach to  operators (\ref{op1}) was suggested  by the first-named co-author in \cite{Ru04a}.
The key idea of \cite{Ru04a} is to use a certain analogue of the stereographic projection  to express (\ref{op1})
through the similar operators on  compact Grassmannians. The latter can be studied using the tools developed in  \cite{GR04}.
This approach enables one to obtain  inversion formulas   for both $R$ and $R^*$  in the framework of Lebesgues spaces for
arbitrary $k'-k>0$ provided that these operators are injective.

The $L^p$-theory of the  operators (\ref{op1}) is of independent interest. The boundedness of these operators  in  $L^p$ spaces with power weights was studied in \cite{Ru14}.

{\bf Aim of the Work and Main Results:} It is known \cite{H11, Ru15} that in the case $k=0$, when $Rf$ becomes the classical
Radon-John transform, the inversion of $R$ can be performed directly, without using stereographic projection.
Note also that the use of the stereographic projection makes all formulas more complicated  because of  inevitable  weight factors.
We wonder, {\it if there is a direct way to invert  the operators (\ref{op1}) under possibly minimal assumptions for functions $f$ and $\vp$}.

New results in this direction are obtained in the present paper. In particular,
we show that   inversion of  (\ref{op1}) can be reduced to consecutive inversion of certain   Radon-John transforms over lower dimensional planes
and  Funk-Radon  transforms on compact Grassmannians. These `simpler' transforms can be inverted by known tools; see, e.g.,  \cite{GR04, H11, Ru04b, Ru13b} and references therein. New subclasses of the so-called {\it quasi-radial functions}, that serve as Grassmannian generalizations of radial functions on $\rn$ and on which operators (\ref{op1}) factorize into the tensor product of known integral geometrical objects, are introduced.

  The paper is organized as follows.  Section 2 contains necessary background related to Radon-like transforms (in affine and compact    settings)   and Erd\'elyi-Kober fractional integrals. These facts are applied in
  Sections 3 and 4   to inversion of the operators (\ref{op1}).

\section{Preliminaries}

\setcounter{equation}{0}

\subsection{Notation.} Let $\agnk$  be the
affine Grassmann manifold of all non-oriented $k$-dimensional
planes  in $\rn$, $0 < k <n$. We denote by $\gnk$  the compact
Grassmann manifold
 of all $k$-dimensional linear subspaces  of $\rn$. Each plane
  $\t \in \agnk$ is  parameterized by the pair
$(\xi, u)$, where $\xi \in \gnk$ and $ u \in \xi^\perp$, the
orthogonal complement to $\xi $ in $\rn$. We denote by $ |\t|$ the
 Euclidean distance of $\t \equiv \t( \xi, u)\in \agnk$ to the origin of $\rn$.
Clearly, $|\t|=|u|$ (the Euclidean norm of $u$). The manifold
$\agnk$ will be endowed with the product measure $d\t=d\xi du$,
where $d\xi$ is the
 $O(n)$-invariant probability measure  on $\gnk$  and $du$ denotes the usual volume element on $\xi^\perp$. We use the notation $C_\mu (\agnk)$ for the  space of continuous function $f$ on $\agnk$ satisfying $f (\t) =O(|\t|^{-\mu})$, $\mu \in \bbr$. The notation $C_\mu (\rn)$ for   the  space of continuous functions on $\rn$ has a similar meaning.  We  also denote
\be\label{lab10}
L^1_\lam (G(n,k)) =\Big \{ \; f : \; \iagr \frac
{|f(\t)| \;
 d\t}{(1+|\t|)^\lam} < \infty  \; \Big \}. \ee

In the following, $S^{n-1} = \{ x \in \bbr^n: \ |x| =1 \}$ is the unit sphere
in $\bbr^n$.   For $\theta \in S^{n-1}$,
$d\theta$ stands for  the surface element on $S^{n-1}$;  $\sigma_{n-1} =  2\pi^{n/2} \big/ \Gamma (n/2)$ is the surface area of $S^{n-1}$.  We set $d_*\theta= d\theta/\sigma_{n-1}$ for  the normalized surface element on $S^{n-1}$.

For $k'>k$ and $\eta \in \gnkp$, we denote by  $G_k (\eta)$ the
Grassmann manifold of all $k$-dimensional linear subspaces of
$\eta$; $ e_1$, $\ldots$, $e_n$ are the coordinate unit
 vectors in $\rn$.  Given $0 < k<k'<n$, we
 use the following notations for the coordinate planes:
 \be\label{mo1} \rk=\bbr e_1 \oplus \cdots \oplus\bbr e_k, \qquad
 \rkp=\bbr e_1 \oplus \cdots \oplus \bbr e_{k'}, \ee
\be\label{mo2} \rkk=\bbr e_{k+ 1}\oplus \cdots \oplus \bbr
 e_{k'}, \qquad \bbr^{n-k}=\bbr e_{k+ 1}\oplus \cdots \oplus \bbr
 e_{n}.\ee
\be\label{mo3} \bbr^{n-k'}=\bbr e_{k'+ 1}\oplus \cdots \oplus \bbr
 e_{n}.\ee
 The letter $c$ stands for a constant that can be different at each
occurrence;  $[ \alpha]$ denotes the integer part of the real number $\a$. All integrals are understood as Lebesgue integrals, unless otherwise stated. We say that  an integral under consideration
 exists in the Lebesgue sense if it is finite when all functions under the sign of integration are replaced by their absolute values.

\subsection{The Radon Transforms for a Pair of Affine Grassmannians.}
 Let $\agnk $ and $ \agnkp$ be a pair of affine Grassmann manifolds of
 $k$-planes $\t$ and $k'$-planes $\z$ in $\rn$,
 respectively; $ \; 1\le k < k' \le n-1$. We write
 \be\label{lab1} \t \!\equiv  \! \t( \xi, u)\in \agnk, \quad  \quad
 \z \! \equiv \! \z(\eta, v) \in \agnkp. \ee
The Radon transform  of a function $f $ on $\agnk$ is a
function $Rf$  on $\agnkp$ defined by
\be\label{lab2} (Rf)(\z) = \intl_{\t
\subset \zeta} f(\t)\, d_\z \t.\ee
In terms of (\ref{lab1}) it means
\be\label{lab2z}
(Rf)(\eta,v)= \intl_{\xi \subset \eta} d_\eta \xi
\intl_{\xi^\perp \cap  \eta} f(\xi, v+x) \,dx,\ee
 where $d_\eta \xi$
denotes the canonical probability measure on the Grassmannian $G_k (\eta)$ of
all $k$-dimensional linear subspaces  of $\eta$. The right-hand
side of (\ref{lab2z}) gives precise meaning to the integral $\int_{\t
\subset \zeta} f(\t)\, d_\z \t$  denoting integration over all
$k$-planes $\t$ in the $k'$-plane $\z$. Assuming $g \in SO(n)$ to
be a rotation satisfying
 \[ g: \rkp \to \eta, \qquad g: e_{k'+1}
\to v/|v|,  \] and setting $f_g(\t)=f(g\t)$, one can write (\ref{lab2})
as
 \be\label{lab3} (Rf)(\eta,v)=\intl_{G_{k',k}} d\sig \intl_{\sig^\perp \cap \rkp}
  f_g (\sig, |v|e_{k'+1}+y) \, dy. \ee

The dual Radon transform  of a function $\vp(\z) \equiv \vp(\eta,
v)$ on $\agnkp$ is a function $(R^*\vp)(\t)\equiv (R^*\vp)(\xi, u)$  on $\gnk$ defined by
\bea\label{lab5} (R^*\vp)(\t)
&=& \intl_{\z \supset \t} \vp(\z)\, d_\t \z\equiv  \intl_{\eta  \supset \xi}
\vp(\eta +u) \, d_\xi \eta\\
&=&\intl_{\eta \supset \xi}
\vp(\eta,\Pr_{\eta^\perp} u) \,d_\xi \eta.\nonumber \eea Here
$\Pr_{\eta^\perp} u$ is the orthogonal projection of $u \;
(\in \xi^\perp)$ onto $\eta^\perp (\subset \xi^\perp)$, $d_\xi
\eta$ is the relevant probability measure.   In order to give (\ref{lab5}) precise meaning, we choose
a rotation $g_\xi \in SO(n)$ so that $g_\xi \rk=\xi$, and let  $K_0\subset SO(n)$ be the isotropy subgroup at
$\bbr^{k}\in \gnk$.
 Then (\ref{lab5}) means \be\label{lab6} (R^*\vp)(\t)\equiv (R^*\vp)(\xi,u)=
\intl_{K_0} \vp(g_\xi \rho \rkp +u) \, d\rho. \ee

\begin{lemma}\label{L1} \cite[Lemma 2.1]{Ru04a}   The equality
\be\label{lab7} \iagrp (Rf)(\z) \vp(\z) \,d \z=\iagr f(\t) (R^*\vp)(\t)\, d\t \ee holds provided that
 the integral in  either side is finite when $f$ and $\vp$ are replaced by $|f|$ and $|\vp|$,
respectively.
\end{lemma}

\begin{lemma} \label{L2} ${}$\hfill

{\rm (i)}If $f \in L^p (\agnk), \; 1 \le p < (n-k)/(k' - k)$,
then $(Rf)(\z)$ is finite for almost all $\z \in \agnkp$. If $ f\in C_\mu (\agnk)$, $\mu >
 k' - k$, then $(Rf)(\z)$ is finite for all $\z \in
 \agnkp$. The conditions $p < (n-k)/(k' - k)$ and $\mu>
 k' - k$ are sharp.

{\rm (ii)}The dual transform  $(R^*\vp)(\t)$ is finite a.e. on $\agnk$ for every locally integrable function $\vp$ on $\agnkp$ and represents a  locally integrable function on $\agnk$.

\end{lemma}

The statement (i) is proved in   \cite[Corollary 2.6]{Ru04a}. The statement (ii) follows from the equality
\be\label{lab8} \intl_{|\t|<a} (R^*\vp)(\t) \, d\t =
 \const \intl_{|\z|<a}\vp (\z) \,(a^2 \! - \! |\z|^2)^{(k' - k)/2} \,d \z  \ee
which is a particular case of the formula (2.19) from  \cite{Ru04a}.

We will also work with weighted spaces (\ref{lab10}).
\begin{lemma} \label{L3} \cite[Proposition  1.2]{Ru04a} For $\lam =n-k'$, the Radon transform
(\ref{lab2}) is a linear bounded operator from $L^1_\lam
(G(n,k))$ to $L^1_{\lam +\del} (G(n,k')),$ $\forall \del >0$.
The exponent $\lam =n-k'$ is best possible.
\end{lemma}

\begin{lemma} \label{L4} \cite[Propositions  1.3, 1.4]{Ru04a}

{\rm (i)} For $f \in L^p (\agnk), \; 1
\le p < (n-k)/(k' -k)$ or $ f \in L^1_{n-k'} (G(n,k))$, the Radon
transform $Rf$  is injective if and only if $ k+k' \le
n-1$.

{\rm (ii)}  For $\vp \in L^1_{k+1}(\agnkp)$, the
dual Radon transform $R^*\vp$  is injective if and only
if $ k+k' \ge n-1$.
\end{lemma}

\subsection{Fractional Integrals and Derivatives on the Half-Line}

More information about fractional integrals in this section can be found in \cite{Ru13b, Ru15}.
Let $f$ be a  function on  $\bbr_+=
(0,\infty)$. For $\alpha
> 0$ and $t>0$, we consider two types of {\it Riemann-Liouville fractional integrals} (left-sided and
right-sided) defined by
\be\label{rl+} (I^\a_{+}f ) (t) \!= \!
\frac{1}{\Gamma (\alpha)} \intl^t_0 \frac{f(r) \,dr} {(t  \!- \! r)^{1-
\alpha}},   \quad  (I^\a_{-}f ) (t)  \!= \! \frac{1}{\Gamma (\alpha)} \intl_t^{\infty}
\frac{f(r) \,dr} {(r \!- \!t)^{1- \alpha}}.\ee
 Fractional derivatives $\D^\a_{\pm}\vp$  of order $\a>0$ are defined as left
inverses of the corresponding fractional integrals, so that
\be\label{09zsew}\D^\a_{\pm}I^\a_{\pm} f=f.\ee
 The operators  $\D^\a_{\pm}$ may have different analytic forms  depending on the class of functions.
For example, if $\alpha = m +
\alpha_0, \;
 m = [\alpha], \; 0 \le \alpha_0 <
1$, then
 \be\label{frr+}\Cal D^\a_{\pm} \vp = (\pm d/dt)^{m +1} I^{1 -
\alpha_0}_{\pm} \vp. \ee
The  existence of the fractional derivative  and the equality (\ref{09zsew}) must be justified at each occurrence.
 The expressions (\ref{frr+})  are called {\it Riemann-Liouville fractional derivatives} of $\vp$.

 We shall also work with the so-called {\it modified Erd\'elyi-Kober fractional integrals}  having
 the form
\bea \label {as34b12}(I^\a_{+, 2} f)(t)&=&\frac{2}{\Gam
(\a)}\intl_0^t (t^2 -r^2)^{\a -1}f (r) \, r\, dr\quad (\text {\rm left-sided}),\\
\label{eci} (I^\a_{-, 2} f)(t)&=&\frac{2}{\Gam
(\a)}\intl_t^\infty (r^2 - t^2)^{\a -1}f (r)  \, r\, dr\quad (\text {\rm right-sided}),\quad
\eea
where $t>0$. Below we review basic facts from \cite[Subsection 2.6.2]{Ru15} related to the existence of these integrals and the corresponding inversion formulas.

\begin{lemma}\label{lifa2}${}$\hfill

{\rm (i)} The integral $(I^\a_{+, 2} f)(t)$ is absolutely convergent for almost all $t>0$ whenever $r\to rf(r)$ is a locally integrable function on $\bbr_+$.

{\rm (ii)} Let $a>0$. If
 \be\label{for10z} \intl_a^\infty |f(r)|\, r^{2\a -1}\, dr
<\infty, \ee
 then $(I^\a_{-, 2} f)(t)$ is finite for almost all $t>a$.
  If  $f$ is non-negative, locally integrable on $[a,\infty)$, and (\ref{for10z}) fails, then $(I^\a_{-, 2} f)(t)=\infty$ for every $t\ge a$.

 \end{lemma}

Fractional derivatives of the  Erd\'elyi-Kober type  are defined as the  left inverses
$ \Cal D^\a_{\pm, 2} = (I^\a_{\pm, 2})^{-1}$.  We have
\be\label {0u8n}
\Cal D^\a_{\pm, 2}\vp = \Lam^{-1}\Cal D^\a_{\pm} \Lam \vp, \qquad    (\Lam f) (t)=f (\sqrt {t}),\ee
where the  Riemann-Liouville  derivatives $\Cal D^\a_{\pm}$ can be
chosen in different forms, depending on our needs. For example, if $\alpha = m +
\alpha_0, \;
 m = [\alpha], \; 0 \le  \alpha_0 <
1$, then, formally, (\ref{frr+}) yields
\be\label{frr+z}   \Cal D^\a_{\pm, 2} \vp=(\pm D)^{m +1}\,
I^{1 - \alpha_0}_{\pm, 2}\vp, \qquad D=\frac {1}{2t}\,\frac {d}{dt}.\ee
This formula is well-justified in the "$+$" case when $\vp=I^\a_{+, 2} f$ with $rf(r)$ being locally integrable on $\bbr_+$; cf. Lemma \ref{lifa2}(i).

Inversion of the operator $I^\a_{-, 2}$ deserves special consideration because
 the analytic expression of $\Cal D^\a_{-, 2}$ essentially depends on the behavior of functions at infinity.

\begin{theorem}\label{78awqe} Let $\vp= I^\a_{-, 2} f$, where $f$  satisfies   (\ref{for10z}) for every $a>0$. Then  $f(t)= (\Cal D^\a_{-, 2} \vp)(t)$ for almost all $t\in \bbr_+$ and   $\Cal D^\a_{-, 2} \vp$ has one of the following forms.

\noindent {\rm (i)} If $\a=m$ is an integer, then
\be\label {90bedr}
\Cal D^\a_{-, 2} \vp=(- D)^m \vp, \qquad D=\frac {1}{2t}\,\frac {d}{dt}.\ee

\noindent {\rm (ii)}   If $\alpha = m +\alpha_0, \; m = [ \alpha], \; 0 < \alpha_0 <1$, then
\be\label{frr+z33} \Cal D^\a_{-, 2} \vp = t^{2(1-\a+m)}
(- D)^{m +1} t^{2\a}\psi, \quad \psi=I^{1-\a+m}_{-,2} \,t^{-2m-2}\,\vp.\ee
 In particular, for   $\a=k/2$, $k$ odd,
 \be\label{frr+z3}  \Cal D^{k/2}_{-, 2} \vp = t\,(- D)^{(k+1)/2} t^{k}I^{1/2}_{-,2} \,t^{-k-1}\,\vp.\ee
Alternatively,
\be\label{frr+z4} \Cal D^\a_{-, 2} \vp=2^{-2\a}\, \Cal D^{2\a}_- \, t\,  I^\a_{-, 2}\, t^{-2\a-1} \, \vp,\ee
where $ \Cal D^{2\a}_-$ denotes the Riemann-Liouville  fractional derivative of order $2\a$ (cf. (\ref{frr+})).

If, moreover, $\int_a^\infty |f(t)|\, t^{2m +1}\, dt
<\infty$ for all $a>0$, then
\be\label{frr+z4y0} \Cal D^\a_{-, 2} \vp=(- D)^{m +1} I^{1-\a+m}_{-, 2}\, \vp.\ee
\end{theorem}

Many other inversion formulas for fractional integrals can be found in \cite[Section 2.4]{Ru15}.

\subsection{Radon-John Transforms}

Given an integer  $0< d<n$, the Radon-John $d$-plane transform of a function $f$ on $\rn$ is a function $R_d f$ on $G(n,d)$ defined by the integral
\be\label{lab20} (R_d f)(\t)=\intl_\t f(x)\,d_\t x,
\ee
where $d_\t x$ stands for the Euclidean volume element of the plane $\t$. The existence of this integral depends on the class of functions $f$.  If $f \in C_\mu  (\bbr^n)$, $\mu>d$, then $(R_d f)(\t)$ is finite for all $\t\in G(n,d)$. If $f \in L^p (\bbr^n)$,  $1\le p<n/d$, then $(R_d f)(\t)$ is finite a.e. on $G(n,d)$. Both restriction $\mu>d$ and $1\le p<n/d$ are  sharp.
We also have the following statement which is a  reformulation of Theorem 3.2 from \cite{Ru13b}.

\begin{lemma}\label{L6}  If
\be \label {lkmuxk}
\intl_{\rn}  \,\frac{|f(x)|} {(1+|x|)^{n-d}} \, dx<\infty,\ee
 then  $(R_d f)(\t)$ is finite for  almost all $\t\in G(n,d)$.
If $f$ is nonnegative, radial, and (\ref{lkmuxk}) fails, then  $ (R_d f)(\t)\equiv \infty$.
\end{lemma}

For $x \in \rn$ and $\t \in G(n,d)$,
we denote by \be r=|x|=\dist (o,x), \qquad s=|\t|=\dist (o, \t)\ee
 the corresponding distances from the origin.

\begin{lemma}\label {L6a} {\rm  (cf. \cite[Lemma 2.1]{Ru04b})} If $f$ is
 a  radial function on $\rn$ satisfying  (\ref{lkmuxk}), then $R_d f$   is
 a  radial function on $ G(n,d)$. Moreover, if
 $f(x) \!\equiv \!f_0(r)$ and $(R_d f)(\t) \equiv F_0(s)$, then
\bea\label{ppaawsdz}
F_0(s)&=&\sig_{d-1}\intl_s^\infty f_0(r)
(r^2 -s^2)^{d/2 -1} r dr, \\
&=&\label{resex1k} \pi^{d/2} \,(I^{d/2}_{-,2} f_0)(s).
 \eea
\end{lemma}

A variety of inversion formulas for the $d$-plane transform are known; see, e.g., \cite {GGG, H11,  Ru04b, Ru13, Ru13b} and references therein.
For example,  the following theorems were proved in \cite{Ru13b}.

\begin{theorem}\label{invr1abf}
A function $f \in C_\mu  (\bbr^n)$,  $\mu>d$, can be recovered from $\vp=R_df$ by the  formula
\be\label{nnxxzz}f(x)=(R_d^{-1}\vp)(x)= \lim\limits_{t\to 0}\, \pi^{-d/2} (\Cal D^{d/2}_{-, 2} F_x)(t), \ee
\be\label{podf}
F_x (t)= \intl_{SO(n)} \!  \vp (\gam \bbr^d +x + t\gam e_n) \,
 d\gam, \ee
where the  limit  is uniform on $\bbr^n$ and the  Erd\'elyi-Kober differential operator
$\Cal D^{d/2}_{-, 2}$ can be computed as follows.

\noindent {\rm (i)} If $d$ is even, then
\be\label {90bedrik}
\Cal D^{d/2}_{-, 2} F_x =(- D)^{d/2} F_x, \qquad D=\frac {1}{2t}\,\frac {d}{dt}.\ee

\noindent {\rm (ii)}  For any $1\le d\le n-1$,
\be\label{frr+z3k} \Cal D^{d/2}_{-, 2} F_x = t^{2-d+2m}
(- D)^{m +1} t^{d}\psi, \quad \psi=I^{1-d/2+m}_{-,2} \,t^{-2m-2}\,F_x,\ee
where $m=[d/2]$. Alternatively,
\be\label{frr+z4k}  \Cal D^{d/2}_{-, 2} F_x = 2^{-d}  \left (-\frac{d}{dt}\right)^d\, t\,  I^{d/2}_{-, 2}\, t^{-d-1} \,F_x.\ee
Under the stronger assumption $\mu >2+2[d/2] \, (>d)$,
 $\Cal D^{d/2}_{-, 2}$ can also be  computed as
\be\label{frr+z4yz}\Cal D^{d/2}_{-, 2} F_x =(- D)^{m +1} I^{1-\a+m}_{-, 2}\,F_x.\ee
\end{theorem}

Note that  powers of $t$ in these formulas stand for the corresponding multiplication operators.

The next theorem contains similar results for $L^p$-functions.

 \begin{theorem}\label{invr1p}
A function $f \in L^p (\bbr^n)$,  $1\le p<n/d$, can be recovered from $\vp=R_df$  at almost every  $x\in \rn$ by the  formula
\be\label{nnxxzz}f(x)=(R_d^{-1}\vp)(x) =  \lim\limits_{t\to 0}\, \pi^{-d/2} (\Cal D^{d/2}_{-, 2}F_x)(t),\ee
where the  limit  is understood in the $L^p$-norm.
Here $F_x$ is defined by (\ref{podf}) and $\Cal D^{d/2}_{-, 2}F_x$ is  computed as in Theorem \ref{invr1abf}, where  (\ref{frr+z4yz})
is applicable under the stronger assumption $1\le p<n/(2+2[d/2])$.
\end{theorem}

\subsection{The Funk-Radon Transforms on Compact Grassmannians}\label{frad}

Let $\gnk$ and $\gnkp$ be a pair of  compact
  Grassmann manifolds of linear subspaces of $\rn$, $0<k<k' <n$. For a function $\Phi$ on $\gnk$, we consider the  Funk-Radon transform\footnote{The subscript "$(n)$" below is intentional. In Section 3 we will be dealing with the similar transform $\F_{(\ell)}$, $\ell <n$, associated with
  Grassmannians  $G_{\ell, k}$ and $G_{\ell, k'}$.}
  \be\label{ku1} (\F_{(n)} \Phi)(\eta)=\intl_{\xi \subset \eta} \Phi(\xi)\, d_\eta \xi, \qquad \eta \in G_{n, k'}\,, \ee
  that integrates $\Phi$ over the set $G_{k}(\eta)$ of all
  $k$-dimensional subspaces of $\eta$ with respect to  the corresponding
  probability measure on $G_{k}(\eta)$.  If $g_{\eta}$ is a
  rotation satisfying $g_{\eta}\bbr^{k'} =\eta$, then
  \be \label{ku2} (\F_{(n)} \Phi)(\eta)=\intl_{G_{k',k}}\Phi(g_{\eta}\xi) \,d\xi.\ee

The dual Funk-Radon transform $\F_{(n)}^* \Psi$ of a
 function $\Psi$ on $G_{n, k'}$ integrates $\Psi$ over the set  of all
  $k'$-dimensional subspaces  $\eta$ containing  the
  $k$-dimensional   subspace $\xi$, namely,
 \be\label{ku3} (\F_{(n)}^* \Psi)(\xi)=\intl_{\eta \supset \xi} \Psi (\eta)\, d_\xi \eta, \qquad \xi \in \gnk. \ee

To give this integral precise meaning, we denote by
$g_{\xi}$ a rotation satisfying $g_{\xi}\bbr^{k}=\xi$  and
let  $K_0\subset SO(n)$ be the isotropy subgroup at
$\bbr^{k}\in \gnk$. Then (\ref{ku3}) is understood as follows:
\be\label{ku4} (\F_{(n)}^* \Psi)(\xi)=\intl_{K_0}\Psi (g_{\xi}\rho \bbr^{k'})\, d\rho
.\ee

\begin{lemma}\label{cri} ${}$\hfill

\noindent {\rm (i)} The Funk-Radon transform $\F_{(n)}$ and its dual $\F_{(n)}^*$ are
linear bounded operators from $L^1 (G_{n, k})$ to $L^1
(G_{n, k'})$, and from $L^1 (G_{n, k'})$ to $L^1 (G_{n,
k})$, respectively.

\noindent {\rm (ii)} The duality relation
\be\label{ksu4} \intl_{G_{n, k'}} (\F_{(n)} \Phi)(\eta)\, \Psi (\eta)\, d\eta=\intl_{G_{n, k}}  \Phi (\xi)\,(\F_{(n)}^* \Psi)(\xi)\, d\xi\ee
holds provided that the integral in either side of this equality is finite when $\Phi$ and $\Psi$ are replaced by $|\Phi|$ and $|\Psi|$, respectively. In particular, for $\Phi \in L^1(G_{n, k})$,
\be\label{ksu41} \intl_{G_{n, k'}} (\F_{(n)} \Phi)(\eta)\, d\eta=\intl_{G_{n, k}}  \Phi (\xi)\, d\xi.\ee
\end{lemma}

We review some facts from \cite{GR04}   in
our notation. Let $\pk$ be the cone  of positive definite
symmetric $k
 \times k$ matrices $r=(r_{i,j})$. The Siegel gamma function associated to  $\pk$
 is defined by  \be\label{ga1}
\Gam_{k}(\a)=\intl_{\pk} e^{-\text{\rm tr} (r)}  \, \det (r)^{\a -d} dr,
 \ee
\[d= (k+1)/2, \qquad dr=\prod_{i \le j} dr_{i,j},\qquad \mbox{\rm tr} (r)= {\hbox{\rm trace of $r$}}.\] This integral converges
 for all $Re \, \a>d-1$ and represents a product of  usual
 gamma functions:
\be\label{ga2} \Gam_{k}(\a)=\pi^{k(k-1)/4} \,\Gam (\a) \,\Gam \left (\a- \frac
{1}{2}\right) \ldots \Gam \left(\a- \frac {k-1}{2}\right). \ee The
G{\aa}rding-Gindikin fractional integral of a function $f$ on
$\pk$ is  defined by
\be\label{ga3} (\ia f)(r) =\frac {1}{\Gam_{k}(\a)}
\intl_0^r f(s) \, \det (r-s)^{\a-d} ds, \quad Re \, \a > d-1, \ee
where $\int_0^r$ denotes integration over the set $\{s: s \in \pk,
\, r-s \in \pk \}$. We  define a differential operator (in the
$r$-variable):
 \bea \quad
\label{ga4} \qquad D_+ =\det \left ( \eta_{i,j} \, \frac {\partial}{\partial
 r_{i,j}}\right ), \quad
  \eta_{i,j}= \left\{
 \begin{array} {ll} 1  & \mbox{if $i=j$,}\\
 1/2 & \mbox{if $i \neq j,$}
\end{array}
\right. \eea so that \[ D_{+}^m I_{+}^\a f= I_{+}^{\a-m} f, \qquad
m
 \in \bbn, \quad Re \, \a  > m+d-1.\] This equality holds pointwise if  $f$ is
 good enough and is understood  in the sense of distributions, otherwise.

 Let $\xi \in \gnk, \; \eta \in \gnkp, \; 1 \le
 k<k'\le n-1$. Given a function $\Psi$ on $\gnkp$,
 we introduce
 the   mean value operator
 \be\label{ga5}
 (M^\ast_r \Psi)(\xi)= \intl_{O(k)} d u\intl_{ \{
\eta \, :  \; {\rm Cos}^2 (\eta,y)=u^Tru \}} \Psi (\eta) \, dm_{\xi} (\eta),
\quad
 r \in \P_{k}, \ee
where $\Cos^2 (\eta,y) \stackrel {\rm def}{=}y^T\Pr_{\eta}y$,
$y$ is a matrix whose columns form
an orthonormal basis of $\xi$. Changing variables  $y \to y\gam$, $u \to u\gam$ with $\gam \in O(k)$, one can readily see that the right-hand side of (\ref{ga5}) is independent of the choice of the orthonormal basis $y$ of $\xi$.
Detailed explanation of the definition of   the matrix-valued $\Cos$-function
 are given in \cite[Section 3]{GR04} and \cite[p. 156]{Zha1}.

\begin{theorem} \cite[Theorem 1.2]{GR04}\label {ifu}
 Let $\Phi \in L^1(\gnk)$.  Suppose that
 \[
 \Psi (\eta)= (\F_{(n)} \Phi)(\eta), \qquad \eta \in G_{n, k'}, \qquad 1 \le
 k<k'\le n-1,\] and denote \[ \a=\frac{k'
-k}{2}, \qquad \tilde \Psi_{\xi} (r) =\det (r)^{\a -1/2}(M^\ast_r
\Psi)(\xi), \qquad \xi \in \gnk.\]
The operator $\F_{(n)}$ is
injective if and only if \be k+k' \le n. \ee Under this
condition, the function $\Phi$ can be recovered by the formula $\Phi
=\F_{(n)}^{-1}\Psi$, where
\be\label{fa1x}(\F_{(n)}^{-1}\Psi) (\xi)=c \lim\limits_{r \to
I_{k}}^{(L^1)} (D_+^m
 I_+^{m-\a} \tilde \Psi_{\xi}) (r),  \qquad c=\frac
{\Gam_{k}(k/2)}{\Gam_{k}(k'/2)}.
\ee
Here $m$ is an arbitrary integer greater than
$(k'-1)/2$, $I_{k}$, is the identity $k \times k$ matrix, and the
differentiation is understood in the sense of distributions. In
particular, for
 $k'-k=2\ell, \; \ell \in \bbn$, we have
\be\label{fa2x} (\F_{(n)}^{-1}\Psi) (\xi)=c \lim\limits_{r \to
I_{k}}^{(L^1)} (D_+^\ell  \tilde \Psi_{\xi}) (r). \ee
 If $\Phi$ is a continuous function on $\gnk$, then the limit in
 (\ref{fa1x}) and (\ref{fa2x}) can be understood in the $\sup$-norm.
\end{theorem}

\begin{remark} {\rm Since the Funk-Radon transform and its dual can be expressed one through another, Theorem \ref{ifu} implies the corresponding inversion result for $\F_{(n)}^*$. In general, every  function $\psi$ on $G_{n,\ell}$ can be regarded as a function on $G_{n,n-\ell}$ if we set $(T\psi) (h^\perp)=\psi (h)$, $h\in G_{n,\ell}$. One can readily show (cf. \cite[Lemma 4.3]{Ru04a}) that
\[\intl_{G_{n,n-\ell}} (T\psi) (h^\perp)\, dh^\perp= \intl_{G_{n,\ell}} \psi (h)\, dh,\]
where $dh$ and $dh^\perp$ are the corresponding probability measures.
Thus  $\F_{(n)}^* =T\tilde \F_{(n)}T$, where  $\tilde \F_{(n)}$ is the Funk-Radon transform for a pair of Grassmannians $G_{n, n-k'}$ and $G_{n, n-k}$. It follows that

\be\label{bm1}
(\F_{(n)}^* )^{-1}=T(\tilde \F_{(n)})^{-1}T.
\ee
}
\end{remark}

\section{Inversion of the Radon Transform for a Pair of Affine Grassmannians}

The main idea is to reduce the problem to known inversion formulas in compact and non-compact settings by making use of a certain factorization procedure. Subsection 3.1  is devoted to Radon transforms of the so-called quasi-radial functions satisfying certain symmetry. The general case is considered in Subsection 3.2.

\subsection{Radon Transforms of Quasi-Radial Functions}\label{sfo}
\begin{definition}\label{qr} A function $f: \agnk \to \bbc$ is  called {\it quasi-radial} if $f(\xi, \cdot)$ is a radial function on $\xi^\perp$ for every $\xi\in \gnk$, in other words, if $f(\xi, u)=f_0(\xi, |u|)$ for some function $f_0$ on $\gnk \times \bbr_+$.
\end{definition}

\begin{lemma} \label{liom} If $f(\xi, u)\equiv f_0(\xi, |u|)$ is a quasi-radial function on $\agnk$ satisfying \be\label{kuku05} \intl_{\gnk} d\xi\intl_{a}^{\infty}|f_0(\xi, r)|  \,r^{k'-k-1} \,dr<\infty\,,\quad \forall \,a>0,\ee
 then the Radon transform  $\vp=Rf$, where
 \be\label{lopi}
(Rf)(\eta, v) \equiv  \intl_{\xi \subset \eta} d_\eta \xi
\intl_{\xi^\perp \cap  \eta} f(\xi, v+x) \,dx\ee
 is  quasi-radial too. Moreover, if  $\vp (\eta, v)= \vp_0 (\eta, |v|)$, then
\be\label{be1}
\vp_0 (\eta, s)=\intl_{\xi \subset \eta} g_\xi(s)\, d_\eta \xi, \ee
where
 \be\label{be2}
  g_\xi(s)=\sig_{k'-k-1}\intl_s^\infty f_0 (\xi, r) \, (r^2 -s^2)^{(k'-k)/2 -1} r\, dr.\ee
\end{lemma}
\begin{proof}
For fixed $\eta$ and $\xi$, the inner integral in (\ref{lopi}) is the Radon-John transform of the function $f(\xi, \cdot): \xi^\perp \to \bbc$ over the
$(k'-k)$-dimensional plane $(\xi^\perp \cap  \eta)+v$ in $\xi^\perp \; (\sim \bbr^{n-k})$. If $f(\xi, \cdot)$ is radial on $\xi^\perp $, the result follows  from Lemmas \ref{L6a} and \ref {lifa2}.
\end{proof}

The right-hand side of (\ref{be1}) is a constant multiple of the tensor product of the Funk-Radon transform (\ref{ku1}) and the Erd\'elyi-Kober-type operator (\ref{eci}).
  Combining Theorems \ref{ifu} and \ref{78awqe}, we obtain the following inversion result.

\begin{theorem} \label{kuku05a} Let $\vp=Rf$, $1\le k<k'\le n-1$, $f(\xi, u)=f_0(\xi, |u|)$, where $f_0$ satisfies (\ref{kuku05}). Suppose that $\vp\equiv \vp_0(\eta,s)\equiv\vp_s(\eta)$; cf.  Lemma \ref {liom}. Then
\bea\label{kuku06}
f_0(\xi, r)\!=\!\pi^{-\a}(D^{\a}_{-, 2}h_\xi)(r), \quad h_\xi(s)\!=\!(\F_{(n)}^{-1}\vp_s)(\xi), \quad \a\!=\!\frac{k'\! -\!k}{2},\eea
where $\F_{(n)}^{-1}$ is defined by (\ref{fa1x}),
\be\label{frr+z3a}
(D^{\a}_{-, 2}h_\xi)(r)=(- D)^{\a} h_\xi , \quad D=\frac {1}{2r}\,\frac {d}{dr},\ee
if $\a$ is an integer, and
 \be\label{frr+z3}  (D^{\a}_{-, 2}h_\xi)(r)=  r\,(- D)^{\a+1/2} r^{2\a}I^{1/2}_{-,2} \,r^{-2\a-1}\,h_\xi,\ee
otherwise\footnote{As in Theorem \ref{78awqe}, powers of $r$ stand for the corresponding multiplication operators.}.
\end{theorem}
\begin{proof}
To  prove this theorem, it remains to note that both Theorems \ref{ifu} and \ref{78awqe} are applicable owing to  (\ref{kuku05}). For example,
 the assumptions of Theorems \ref{ifu} are satisfied because  $g(\cdot, s)\in L^1(\gnk)$. Indeed,
 \[ \intl_{\gnk} |g(\xi, s)|\, d\xi\le \const \intl_s^\infty (r^2 -s^2)^{\a -1} r\, dr\intl_{\gnk} |f_0(\xi, r)|\, d\xi<\infty\]
 for almost all $s>0$ according to Lemma \ref{lifa2}(ii) and  (\ref{kuku05}).
\end{proof}

\subsection{The  General Case}

 Our approach is inspired by   Gonzalez and Kakehi \cite[pp. 255, 258]{GK03}  who applied the   Fourier transform to $(Rf)(\eta, v)$ in the $v$-variable.  However, the use of the Fourier transform  leads to inevitable restrictions on the class of functions. As we shall see below, these restrictions can be essentially weakened
 if  the Fourier transform is replaced  by the suitable Radon-John transform, the $L^p$ theory of which is well-developed; see, e.g., \cite{Ru04b, Ru13b}.

Let  $\vk$ be an integer satisfying
\be\label{lab3.17}
k\le \vk <n-k'\ee
(the role of $\vk$  will become clear from the reasoning below) and  consider an  auxiliary function
\be\label{fa1}
g_{h,\a}(\xi) \equiv g(h,\xi,\a)=\intl_{\xi^\perp \cap  h} f(\xi, y+\a) \,dy\ee
  on the set of triples
\be\label{fa1pp} \Om =\{ (h,\xi,\a) : \;h\in G_{n,k'+\vk}, \;\xi \in G_{k} (h), \;\a \in h^\perp\}.\ee
 We recall that $f$ is a function on $\gnk$, $0<k<k'<n$, and note that  $\dim (\xi^\perp \cap  h)>0$ because  $\dim (\xi^\perp)+ \dim (h)=n-k+k'+\vk >n$.

Our first step is  to reconstruct $g(h,\xi,\a)$ on $\Om$ from $(Rf)(\eta, v)$. The second step is  to reconstruct
  $f(\xi, u)$ on $\agnk$ from $g(h,\xi,\a)$.

 STEP 1.  We replace the target space $G(n,k')$ by the ``bigger'' Grassmannian $ G (n,k'+\vk)$ of $(k'+\vk)$-dimensional planes
 $\tilde \z\equiv \tilde \z(h,\a)$ in $\rn$ and consider the corresponding Radon transform
 \bea (\tilde Rf)(\tilde \z)&\equiv& (\tilde Rf)(h,\a)=\intl_{\xi \subset h} d_h \xi
\intl_{\xi^\perp \cap  h} f(\xi, y+\a) \,dy\nonumber\\
\label {KIOL} &=&\intl_{G_k (h)} g_{h,\a}(\xi) \, d_h \xi, \quad \tilde \z\in G(n, k'+\vk).\eea

 \begin{lemma}\label{STEP1} ${}$ \hfill

 \noindent {\rm (i)} If
 \be\label{lab3.3}
f \in L^p (\agnk), \qquad  1 \le p < \frac{n-k}{k'-k+\vk}, \ee
then $(\tilde R|f|)(\tilde \z)$ is finite for almost all $\tilde \z \in G(n, k'+\vk)$.
For all such  $\tilde \z \equiv \tilde \z(h,\a)$, we have $g_{h,\a} \in L^1 (G_k(h))$ and for almost all $\eta \in G_{k'}(h)$,
\be\label{lm}
\intl_{\xi \subset \eta} g_{h,\a}(\xi) \, d_\eta \xi=\intl_\pi (Rf)(\eta, v)\,d_\pi v. \ee
Here $\pi=(\eta^\perp \cap h) +\a$ is a $\vk$-dimensional affine plane in $\eta^\perp$ parallel to $h$ and
$Rf$ is the Radon transform (\ref{lopi}).

 \noindent {\rm (ii)} If  $f\in C_\mu (\agnk)$, $\mu >  k'-k +\vk$, then $(\tilde R|f|)(\tilde \z)$ is finite for  all $\tilde \z \in G(n, k'+\vk)$, $g_{h,\a}$ is a continuous function on $G_k(h)$, and (\ref{lm}) holds for  all $\eta \in G_{k'}(h)$.
\end{lemma}
\begin{proof}  {\rm (i)} The first statement follows from  Lemma \ref{L2}(i), in which $k'$ should be replaced by  $k'+\vk$. This also gives
 $g_{h,\a} \in L^1 (G_k(h))$ for almost all $\tilde \z \equiv \tilde \z(h,\a)\in G(n,k'+\vk)$.  To prove (\ref{lm}), we formally have
\be\label{lm5c}
\intl_\pi (Rf)(\eta, v)\,d_\pi v=  \intl_{\xi \subset \eta} d_\eta \xi
\intl_{\xi^\perp \cap  \eta} dx\intl_{\eta^\perp \cap  h} f(\xi, x+v_1+\a) \,dv_1.\ee
Keeping in mind that $\xi \subset \eta$ and therefore $\eta^\perp \subset \xi^\perp$, we observe that every vector $V$ in $\xi^\perp$ can be decomposed as
\[V=\Pr_{h^\perp} V + \Pr_{\xi^\perp \cap  \eta} V + \Pr_{\eta^\perp \cap  h} V.\]
Hence the right-hand side of (\ref{lm5c})  can be written as
\be\label{235}\intl_{\xi \subset \eta} d_\eta \xi\intl_{\xi^\perp \cap  h} f(\xi, y+\a) \,dy=\intl_{\xi \subset \eta} g_{h,\a}(\xi) \, d_\eta \xi,\ee
as desired. To make this reasoning rigorous, we observe that for the integral (\ref{KIOL}) with a non-negative function $f$ satisfying (\ref{lab3.3})
we have \be \label {KIOLa} (\tilde Rf)(h,\a)=\intl_{\eta \subset h}  \left  [ \, \intl_{\xi \subset \eta} d_\eta \xi
\intl_{\xi^\perp \cap  h} f(\xi, y+\a) \,dy\right ]\, d_h \eta<\infty\ee
for almost all $ \tilde \z(h,\a)\in G(n,k'+\vk)$;  cf. (\ref{ksu41}). It follows that the expression in square brackets is finite for almost all  $\eta \in G_{k'}(h)$. But this expression is exactly the left-hand side of (\ref{235}). This completes the proof for the $L^p$- case.

 {\rm (ii)}  If   $f\in C_\mu (\agnk)$ with $\mu >  k'-k +\vk$, then $(\tilde R|f|)(\tilde \z)$ is finite for  all $\tilde \z \in G(n, k'+\vk)$ by Lemma \ref{L2} (i). Moreover,  $f\in  L^p (\agnk)$ whenever
\[\frac{n-k}{\mu} \le p < \frac{n-k}{k'-k+\vk},\]
and therefore all statements in (i) remain true. The continuity of $g_{h,\a}$ follows immediately from its definition (\ref{fa1}).  The validity of (\ref{lm})  for  all $\eta \in G_{k'}(h)$ is a consequence of continuity of all functions in this equality.
\end{proof}

Lemma \ref{STEP1} allows us to reconstruct $g(h,\xi,\a)$  on $\Omega$ from $(Rf)(\eta, v)$.   Indeed, denote
\be \label{lm5c1}  G_{h,\a}(\eta)=\intl_\pi (Rf)(\eta, v)\,d_\pi v, \qquad \pi=(\eta^\perp \cap h) +\a.\ee
By (\ref{lm}),
\be\label{lmn}
\intl_{\xi \subset \eta} g_{h,\a}(\xi) \, d_\eta \xi= G_{h,\a}(\eta), \qquad \eta \in  G_{k'}(h).\ee
The left-hand side is the Funk-Radon transform  of $g_{h,\a}$ for a pair of compact Grassmannians $G_{k}(h)$ and $G_{k'}(h)$. Let
\[ \ell=k'+\vk, \qquad \bbr^\ell=\bbr e_1 \oplus \cdots \oplus \bbr e_\ell, \]
and let $\F_{(\ell)}$ be  the Funk-Radon transform for a pair of compact Grassmannians $G_{\ell,k}$ and $G_{\ell,k'}$; cf.  (\ref{lab2}) for $\ell=n$. If $\gam_h$ is a rotation that takes $\bbr^\ell$ to $h$, then (\ref{lmn}) can be written as
\[ \F_{(\ell)} [g_{h,\a} \circ \gam_h]= G_{h,\a}\circ \gam_h.\]
Because $g_{h,\a} \in L^1 (G_k(h))$, it follows that $g_{h,\a} \circ \gam_h \in L^1 (G_{\ell,k})$, and the operator $\F_{(\ell)}$ can be explicitly inverted by Theorem \ref{ifu}. This gives
\be\label{235b}
g_{h,\a}= \F_{(\ell)}^{-1}[G_{h,\a}\circ \gam_h] \circ \gam_h^{-1},\ee
where $\F_{(\ell)}^{-1}$ is defined by (\ref{fa1x})  with $n$ replaced by $\ell=k'+\vk$.  Thus we have proved the following
\begin{proposition} \label {pwo} Let $0<k<k'<n$, $k\le \vk<n-k'$. If
 $f \in L^p (\agnk)$, $1 \le p < (n-k)/(k'-k+\vk)$, then, for almost all $\tilde \z(h,\a)\!\in \!G(n, k'\!+\!\vk)$, the function $g_{h,\a}$  from (\ref{fa1}) can be expressed through the Radon transform $Rf$ by the formula (\ref{235b}) in which $G_{h,\a}$ is defined by (\ref{lm5c1}). If  $f\in C_\mu (\agnk)$, $\mu >  k'-k +\vk$, the  above statement holds for all $\tilde \z(h,\a)\!\in \!G(n, k'\!+\!\vk)$.
\end{proposition}

STEP 2. Our next aim is  to reconstruct
  $f (\t)\equiv f(\xi, u)$ on $\agnk$ from $g(h,\xi,\a)$. To this end, we  fix  $\xi \in \gnk$ and choose a plane
$\tilde \z(h,\a)\in G (n,k'+\vk)$, so that $h\supset\xi$  and $g(h,\xi,\a)$ can be found by Step 1.  Every  triple $(h,\xi,\a)$ obtained this way belongs to the set (\ref{fa1pp}) and determines a plane $\om$ in  $\xi^\perp$ so that
\be\label{k1} \om=(h\cap \xi^\perp)+\a, \qquad \dim (\om)= k' -k +\vk  \overset{\rm def} {=}k_1.\ee

Note that  every $k_1$-plane $\om$ in  $\xi^\perp$ is uniquely represented in the form (\ref{k1}) with some $h \in G_{n, k'+\vk}$ containing $\xi$ and some $\a\in  h^\perp$.
Indeed, we can write $\om =\om_0 +\a$, where  $\om_0\in G_{k_1}(\xi^\perp)$ is parallel to $\om$ and $\a\in \om_0^\perp \cap \xi^\perp$.
 It remains to set $h=\om_0\oplus \xi$ and note that $h^\perp=(\om_0\oplus \xi)^\perp=\om_0^\perp \cap \xi^\perp \ni \a$.

Now we set $f_\xi (u)\equiv f(\xi, u)$ and  write (\ref{fa1}) as
\be\label{fa14}
g(h,\xi,\a) \equiv \intl_{\xi^\perp \cap  h} f_\xi (y+\a) \,dy \overset{\rm def} {=}G_\xi (\om),\qquad \om=(h\cap \xi^\perp)+\a.\ee
 This is the Radon-John $k_1$-plane transform $(R_{k_1} f_\xi)(\om)$
of the function $f_\xi \equiv f(\xi, \cdot)$ defined on $\xi^\perp$.
 The operator $R_{k_1}$ can be explicitly inverted by making use of Theorems \ref{invr1abf} and
  \ref{invr1p} adjusted to our case. Specifically,
 we replace the ambient space $\rn$ in (\ref{nnxxzz})   by $\xi^\perp$, $d$ by $k_1$, $f(x)$ by $f_\xi (u)$, and the average $F_x (t)$  (cf. (\ref{podf})) by
  \be\label{podfq}
F_{\xi,u} (t)= \intl_{SO(\xi^\perp)} \!  G_\xi (\gam \om_{\xi, t} +u) \,
 d\gam. \ee
 Here $SO(\xi^\perp)\subset SO(n)$ is the isotropy subgroup of $\xi^\perp$, $\om_{\xi, t}$ is an arbitrary $k_1$-plane in $\xi^\perp$ at distance $t>0$ from the origin.

 The above reasoning  leads to the following statement.
 \begin{proposition} \label {pwo1}  Let $0<k<k'<n$, $k\le \vk<n-k'$. If
 $f \in L^p (\agnk)$, $1 \le p < (n-k)/(k'-k+\vk)$, then for almost all $\t \equiv \t(\xi, u)\in \agnk$, the function
 $f(\t)\equiv f_\xi (u)$ can be reconstructed from $g(h,\xi,\a)\equiv G_\xi (\om)$,  $\om=(h\cap \xi^\perp)+\a \in G(\xi^\perp, k_1)$, by the formula
 \be\label{nnjj} f_\xi (u)=  \lim\limits_{t\to 0}\, \pi^{-k_1/2} (\Cal D^{k_1/2}_{-, 2}F_{\xi,u})(t),\ee
where,  for almost all $\xi \in \gnk$, the  limit  is understood in the $L^p (\xi^\perp)$-norm.
Here $F_{\xi,u}$ is defined by (\ref{podfq}) and $\Cal D^{k_1/2}_{-, 2}F_{\xi,u}$ is  computed as in Theorem \ref{invr1abf}. If  $f\in C_\mu (\agnk)$, $\mu >  k'-k +\vk$, the  above statement holds for all $\xi \in \gnk$ and the  limit in (\ref{nnjj})  is understood in the $sup$-norm on $\xi^\perp$.
\end{proposition}

Combining Propositions  \ref{pwo} and  \ref{pwo1},  we obtain the following result.

\begin{theorem} \label{mt3} Let $0<k<k'<n$, $k\le \vk<n-k'$. Suppose that
\[  f\in C_\mu (\agnk), \qquad \mu >  k'-k +\vk,\]
or
\[   f \in L^p (\agnk), \qquad 1 \le p < (n-k)/(k'-k+\vk).\]
Then
 $f(\t)\equiv f_\xi (u)$ can be reconstructed for all or almost all $\t \equiv \t(\xi, u)$, respectively, by the formulas (\ref{nnjj}), (\ref{podfq}), and (\ref{fa14}), where  $g(h,\xi,\a) \equiv g_{h,\a} (\xi)$ is determined by Proposition  \ref{pwo}.
\end{theorem}

\begin{remark} In the case $\vk=k$, Theorem \ref{mt3}  holds under the assumptions $1 \le p < (n-k)/k'$ and $\mu >  k'$.
\end{remark}

\section{Inversion of the Dual Radon Transform for a Pair of Affine Grassmannians}

Let us consider the dual Radon transform (\ref{lab5}) which is well defined for every locally integrable function $\vp$ on $\agnkp$ and represents a  locally integrable function  on $\agnk$  by the formula
\be\label{lab64} (R^*\vp)(\t) \equiv (R^*\vp)(\xi, u)=
\intl_{SO(n-k)} \vp(g_\xi \rho \rkp +u) \, d\rho; \ee
see Lemma \ref{L2}.
Our aim  is to reconstruct $\vp$ from $R^*\vp$.

\subsection{The Dual Radon Transform of Quasi-Radial Functions}

In this subsection we consider the case  when $\vp$ is a locally integrable quasi-radial function on $\agnkp$.
By Definition \ref{qr}, there exists a function $\vp_0\equiv\vp_0(\eta, t)$  on $\gnkp\times \bbr_+$ such that
$\vp(\eta, v)=\vp_0(\eta, |v|)$  and
\be\label{lab64x1} \intl_0^a t^{n-k' -1}\, dt\intl_{\gnkp} |\vp_0(\eta, t)|\, d\eta <\infty \quad \forall \, a>0.\ee

Our aim is to  show that averaging $(R^*\vp)(\xi, u)$ over all $u\in\xi^\perp$ at a distance $r$ from the origin
yields a decomposition of $R^*\vp$ into a tensor product of two invertible operators, namely, the Erd\'elyi-Kober type
operator and the dual Funk-Radon transform for a pair of compact Grassmannians.
We will be using the same  notations for the coordinate planes as in (\ref{mo1})-(\ref{mo3}).

Let $g_\xi\in SO(n)$ be a rotation  satisfying  $g_\xi \bbr^k =\xi$. We set
\be
(M_{\xi}f)(r)=\intl_{SO(n-k)}(f\circ g_{\xi})(\rk, r\g e_{n})\,d\g, \qquad r>0,
\ee
where $SO(n-k)$ is the isotropy subgroup  of the coordinate plane $\bbr^{n-k}$; see (\ref{mo2}).
Using (\ref{lab64}) and setting $\vp_\xi =\vp\circ g_\xi$, we formally obtain
\bea\label{m1}
(M_{\xi}R^*\vp)(r)&=&\intl_{SO(n-k)} ( R^*\vp \circ g_\xi)(\rk, r\g e_{n})\,d\g\nonumber\\
&=&\intl_{SO(n-k)} d\g\intl_{SO(n-k)} \vp_\xi (\rho \rkp +  r\g e_{n})\, d\rho.\nonumber\eea
This gives
\bea
(M_{\xi}R^*\vp)(r)&=&\intl_{SO(n-k)} d\g\intl_{SO(n-k)} \vp_\xi (\rho \rkp,  r\Pr_{\rho \rnkp}\g e_{n})\, d\rho\nonumber\\
&=&\intl_{SO(n-k)} d\rho\intl_{SO(n-k)} \vp_\xi (\rho \rkp,  r\rho\Pr_{\rnkp}\rho^T\g e_{n})\, d\g\nonumber\\
&=&\frac{1}{\sig_{n-k}}\intl_{SO(n-k)} d\rho \intl_{S^{n-k-1}}\vp_\xi (\rho \rkp,  r\rho\Pr_{\rnkp}\sig)\,d\sig,\nonumber
\eea
where $S^{n-k-1}$ denotes the unit sphere in  $\bbr^{n-k}$ and $\Pr_{\rnkp}$ stands for the orthogonal projection onto $\rnkp$. Passing to bi-spherical coordinates
\be \sig=\a \cos \psi +\b \sin\psi, \qquad \a\in S^{k'-k-1}, \quad \b\in S^{n-k'-1},\ee
where $ S^{k'-k-1}$ and $S^{n-k'-1}$ denote the unit spheres in $\rkk$ and $\rnkp$, respectively (see, e.g. \cite[Lemma 1.38]{Ru15}), we continue
\bea
(M_{\xi}R^*\vp)(r)&=&\frac{1}{\sig_{n-k}}\intl_{SO(n-k)} d\rho \intl_0^{\pi/2}\cos^{k'-k-1} \psi \, \sin^{n-k'-1} \psi\, d\psi\nonumber\\
&\times&\intl_{S^{k'-k-1}} \,d\a
\intl_{S^{n-k'-1}}\vp_\xi (\rho \rkp,  (\rho\b) r\sin\psi)\,d\b\nonumber\\
&=&\frac{\sig_{k'-k-1}\sig_{n-k'-1}}{\sig_{n-k}}\intl_0^1 s^{n-k'-1}(1-s^2)^{(k'-k)/2-1}\,ds\nonumber\\
\label{m1}&\times&\intl_{SO(n-k)}\vp_\xi (\rho \rkp,  (\rho e_n) \,rs)\,d\rho.
\eea
Changing variables, we obtain
\bea
(M_{\xi}R^*\vp)(r)&=&
\frac{\sig_{k'-k-1}\sig_{n-k'-1}}{r^{n-k-2}\,\sig_{n-k}}\intl_0^r t^{n-k'-1} (r^2-t^2)^{(k'-k)/2-1}\,dt\nonumber\\
\label{m3} &\times&\intl_{\eta\supset\xi}\vp(\eta, t g_\eta e_n) \,d_\xi\eta,\eea
 $g_\eta$ being an orthogonal transformation that takes $\bbr^{k'}$ to $\eta$.
 Since $\vp$ is quasi-radial, it follows that $\vp(\eta, t g_\eta e_n)\equiv \vp_0 (\eta, t)$. Then the integral in  (\ref{m3}) is the dual Funk-Radon transform (\ref{ku3}) of  $\vp_0 (\cdot, t)$ and we can write
\be\label{am3}
(M_{\xi}R^*\vp)(r)= \frac{c}{r^{n-k-2}}\,(\F_{(n)}^* \Psi_r)(\xi), \qquad c=\frac{\pi^{(k'-k)/2}\,\sig_{n-k'-1}}{\sig_{n-k}},\ee
\be\label{am3xd}
\Psi_{r}(\eta)=(I_{+,2}^{(k'-k)/2}  [\psi_{(\cdot)}(\eta)])(r), \qquad \psi_t(\eta)= t^{n-k'-2} \vp_0 (\eta, t).\ee

To make the formal derivation of (\ref{am3}) rigorous, we need to show that  the right-hand side of  (\ref{am3}) exists in the Lebesgue sense when $\vp$ is replaced by  $|\vp|$. To this end, it suffices to assume that $\vp(\eta, v)\equiv\vp_0(\eta, |v|)$ is nonnegative and prove the following statements.

(a) The function $t \to t\psi_{t}(\eta)$ is locally integrable on $\bbr_+$   for almost all $\eta \in \gnkp$;

(b) $\Psi_{r}\in L^1 (\gnkp)$ for almost all $r>0$.

The statement (a) follows from (\ref{lab64x1}). To prove (b), owing to (\ref{am3xd}), we have
\be\label{anxd} \intl_{\gnkp} \Psi_{r}(\eta)\, d\eta=(I_{+,2}^{(k'-k)/2} \tilde\psi)(r),\ee
where $\tilde \psi(t)=t^{n-k'-2} \int_{\gnkp} \vp_0 (\eta, t)\, d\eta$. By (\ref{lab64x1}),  $t\tilde \psi(t)$ is locally integrable on $\bbr_+$ and therefore, by Lemma \ref{lifa2}(i), the expression (\ref{anxd}) is finite for almost all $r>0$. This gives (b).

The  equalities (\ref{am3}) and (\ref{am3xd}) reduce the inversion problem for $R^*\vp$ to inversion of $\F_{(n)}^*$ and $I_{+,2}^{(k'-k)/2}$.
\begin{theorem}  Let $f=R^* \vp$, where
$\vp\equiv \vp_0 (\eta, |v|)$ is a locally  integrable quasi-radial function on
 $\agnkp$, $0<k<k'<n$, $k+k'\ge n$. Then  for almost all $\eta \in \gnkp$ and  almost all $t>0$,
\bea
 \vp_0 (\eta, t)&=&c^{-1} t^{k'+2 -n}(\Cal D_{+,2}^{(k'-k)/2} [\Psi_{(\cdot)}(\eta)])(t), \nonumber\\
 \Psi_{r}(\eta)&=& r^{n-k-2} ((\F_{(n)}^*)^{-1} [(M_{(\cdot)}f)(r)])(\eta),\nonumber\eea
where $c$ is a constant from (\ref{am3}),  $\Cal D_{+,2}^{(k'-k)/2}$ and $(\F_{(n)}^*)^{-1}$ being defined by  (\ref{frr+z})  and  (\ref{bm1}), respectively.
\end{theorem}

\begin{remark} It is known that both  $R$ and $R^*$ take radial functions to radial ones on the corresponding affine Grassmannians \cite{Ru04a}. For quasi-radial functions the situation is different. In Subsection \ref{sfo} it was shown that the operator $R$ takes quasi-radial functions on $\agnk$ to quasi-radial functions on $\agnkp$. A similar result for  $R^*$ may not be true. Let, for example, $k=1$ and $k'=n-1$. Then  (\ref {lab6})
 with $\xi=\bbr^1$ and $ u=e_m$, $m>1$, yields
 \bea (R^*\vp)(\bbr^1, e_m)
&=&\intl_{SO(n-1)} \vp(\rho e_n^\perp+\Pr_{\rho e_n} e_m) \, d\rho\nonumber\\
&=&\intl_{SO(n-1)} \vp((\rho e_n)^\perp +(\rho e_n)(\rho e_n \cdot e_m)) \, d\rho\nonumber\\
&=&\intl_{S^{n-2}} \vp(\sig^\perp +\sig(\sig\cdot e_m)) \, d_*\sig, \nonumber\eea
where $S^{n-2}$ stands for the unit sphere in $e_1^\perp$. Now we choose $\vp=\vp_0$, where $\vp_0 (\eta, v)= |v\cdot e_2|$. The orthogonal complement to $\eta$ is one-dimensional and therefore, the only orthogonal transformation that keeps $\eta$ fixed is the reflection $v \to -v$.  It follows that $\vp_0$ is quasi-radial. Assuming, for simplicity, $n=3$, we have the following  expressions for $m=2$ and $m=3$:
\[(R^*\vp_0)(\bbr^1, e_2)=\intl_{S^{1}} (\sig\cdot e_2)^2 \, d_*\sig=\frac{1}{2\pi}\intl_0^{2\pi} \cos^2 \th \, d\th=\frac{1}{2},\]
\[(R^*\vp_0)(\bbr^1, e_3)=\intl_{S^{1}} |(\sig\cdot e_2)(\sig\cdot e_3)| \, d_*\sig=\frac{1}{2\pi}\intl_0^{2\pi} |\cos \th \, \sin \th|\, d\th=\frac{1}{\pi}.\]
Thus $(R^*\vp_0)(\bbr^1, e_2)\neq (R^*\vp_0)(\bbr^1, e_3)$, which means that $R^*\vp_0$ is not quasi-radial.
\end{remark}

\subsection{The Dual Radon Transform. The general case}

By Lemma \ref{L4},
the dual Radon transform $f=R^*\vp$ that takes  functions $\vp \in L^1_{k+1}(\agnkp)$ to functions  on $\agnk$ is injective if and only
if $ k+k' \ge n-1$. To reconstruct $\vp$ from $f$, we use some ideas from \cite[Section 4]{Ru04a} according to which  $  R^*$  expresses through a certain auxiliary Radon transform $\fr R$  that takes  functions on  $G(n, n-k'-1)$  to functions  on $G(n, n-k-1)$.  The new transform $\fr R$ can be explicitly inverted, e.g.,  as in  Section 3. Thus we shall arrive at explicit inversion of $R^*$.

Let us proceed to details.
For $\t\equiv\t(\xi, u) \in \agnk$, with $ u \neq 0$, we denote by $\{\t\}  \in G_{n,
k+1}$ the smallest linear subspace containing $\t$, and set
\[ \tilde\xi=\{\t\}^{\perp} \in G_{n, n-k-1}, \quad \tilde u= -\frac{u}{|u|^2} \in \tilde\xi^\perp, \quad
\tilde \t\equiv \tilde \t(\tilde\xi,\tilde u)\in G (n, n-k-1). \]
Consider the Kelvin-type mapping
\be \agnk \ni \t \xrightarrow {\quad \nu \quad } \tilde \t \in G(n, n-k-1).\ee
Clearly, $\nu (\nu (\t))=\t$ and $|\t|=|\tilde \t|^{-1}$ (see Notation). In a similar way, for $\z\equiv\z(\eta, v)\in \agnkp$,  $v\neq 0$, we denote
\[ \tilde\eta=\{\z\}^{\perp} \in G_{n, n-k'-1}, \quad \tilde v= -\frac{v}{|v|^2} \in \tilde\eta^\perp, \quad
\tilde \z\equiv \tilde \z(\tilde\eta,\tilde v)\in G (n, n-k'-1), \]
so that
\be \agnkp \ni \z \xrightarrow {\quad \nu \quad } \tilde \z \in G(n, n-k'-1).\ee

\begin{definition} Let  $R: f(\t) \to (Rf)(\z)$ be the Radon
transform (\ref{lab2}) that takes functions on $\agnk$ to functions on $\agnkp, \; k'>k$. If $\tilde \t =\nu (\t)\in G(n, n-k-1)$ and $\tilde \z =\nu (\z)\in G(n, n-k'-1)$, then the associated  Radon
transform $ \tilde f(\tilde\z) \to (\fr R \tilde f)(\tilde \t)$ that integrates $\tilde f $ over all   $\tilde\z$ in $\tilde\t$
 is called {\it
quasi-orthogonal} to $R$.
\end{definition}

\begin{theorem}\label{exath}   \cite[Theorem 5.5]{Ru04a}  Let $0\le k<k' <n$. For a function $\vp$   on $ \agnkp$, we denote
\be\label {hat} (A\vp)(\tilde \z)=|\tilde \z|^{k-n}\vp
(\nu^{-1}(\tilde \z)), \qquad \tilde \z \in G(n, n-k'-1). \ee

{\rm (i)} The following relation holds
\be\label {ild} \intl_{\agnkp}\frac{\vp (\z)
\; d\z}{(1+|\z|^2)^{(k
  +1)/2}}= \frac{\sig_{n-k' -1}}{\sig_{k'}}\intl_{G(n, n-k'-1)}\frac{(A\vp)(\tilde \z)
\; d\tilde \z}{(1+|\tilde \z|^2)^{(k
  +1)/2}} \ee
provided that either side of this equality exists in the Lebesgue sense.

{\rm (ii)} If at least one of the integrals in (\ref{ild}) is finite, then
 \be \label {o9i} (R^* \vp)(\t)=c \, |\t|^{k'
-n}(\fr R A\vp)(\nu (\t)), \qquad c=\frac{\sig_{n-k'
-1}}{\sig_{n-k -1}}. \ee
\end{theorem}

Theorem \ref{exath} paves the way to reconstruction of $\vp$ from $f=R^* \vp$.  Indeed, we write (\ref{o9i}) as
$(\fr R A\vp)(\nu (\t))=c^{-1} |\t|^{n-k'}(R^* \vp)(\t)$ or
\be
(\fr R A\vp)(\tilde \t))=c^{-1} |\tilde \t|^{k'-n}\,f(\nu^{-1}(\tilde\t)).\ee
Inverting $\fr R$ as in Section 3, we formally obtain
\be\label {mku}
\vp (\z)\!=\!|\z|^{k-n} (\fr R^{-1} f_1)(\nu (\z)), \quad f_1 (\tilde \t)\!=\!c^{-1} |\tilde \t|^{k'-n}\,f(\nu^{-1}(\tilde\t)).\ee

To make this reasoning rigorous, we need to choose  a suitable class of  functions $\vp$ that guarantees applicability of (\ref{o9i}) and
  Theorem  \ref{mt3}.

For $1\le p<\infty$, we denote
\be\label {mmat1} \tilde L^p (\agnkp) =\left \{ \vp:\;\intl_{\agnkp} |\z|^{(n-k)p-n-1} |\vp (\z)|^p \, d\z<\infty \right \}.\ee
For $\mu \in \bbr$, let  $\tilde C_\mu (\agnkp)$ be the space of all functions $\vp$ which are
continuous  on the set of all $k'$-planes $\z \subset \rn$ not passing through the origin and satisfy the following condition:
\be\label{dit}
\left\{ \!
 \begin{array} {ll} |\z|^{n-k-\mu}\vp(\z)=O(1) & \mbox{\rm if  $|\z|\to 0$,}\\
 ${}$\\
  |\z|^{n-k}\vp(\z) \to \const   & \mbox{\rm if  $|\z|\to \infty$.}\\
   \end{array}
\right.\ee

The  choice of these classes of functions is motivated by the following lemma.

\begin{lemma}\label{atu}${}$ \hfill

{\rm (i)} For any  $1\le p<\infty$, the relations $\vp \in \tilde L^p (\agnkp)$ and  $A\vp \in L^p(G(n, n-k'-1))$  are equivalent.
 If $1\le p<(k'+1)/(k' -k)$, then the Radon transform $\fr R A\vp$ exists in the Lebesgue sense and
\be\label {kin}
\intl_{\agnkp}\frac{|\vp (\z)|
\; d\z}{(1+|\z|^2)^{(k
  +1)/2}} <\infty.\ee

{\rm (ii)}  For any   $\mu \in \bbr$, the relations $\vp \in \tilde C_\mu (\agnkp)$ and  $A\vp \in  C_\mu (G(n, n-k'-1))$  are equivalent. If $\mu> k' -k$, then the Radon transform $(\fr R A\vp)(\tilde \t)$ is finite for every $\tilde \t \in G(n, n-k-1)$ and (\ref {kin}) holds.
\end{lemma}

\begin{proof}  (i)  To prove the first statement, we observe that
\[ ||A\vp||_p^p =\!\!\intl_{G(n, n-k'-1)}\!\! \! |\tilde \z|^{(k-n)p}\, |\vp(\nu^{-1}(\tilde \z))|^p\, d\tilde \z =\!\!\intl_{G(n, n-k'-1)}\!\frac{(A\psi)(\tilde \z)
\; d\tilde \z}{(1+|\tilde \z|^2)^{(k+1)/2}}, \]
where
\[ \psi(\z)= (1+|\z|^2)^{(k+1)/2}|\z|^{(n-k)p-n-1}\,  |\vp(\z)|^p.\]
Hence, by (\ref{ild}),
\bea
\frac{\sig_{n-k' -1}}{\sig_{k'}} \,||A\vp||_p^p &=& \intl_{\agnkp}\frac{\psi (\z)
\; d\z}{(1+|\z|^2)^{(k  +1)/2}}\nonumber\\
&=&\intl_{\agnkp} |\z|^{(n-k)p-n-1} |\vp (\z)|^p \, d\z,\nonumber\eea
as desired.  The existence of $\fr R A\vp$  follows from Lemma \ref{L2}; (\ref {kin}) holds by H\"older's inequality.

(ii)  The proof of the first statement and  the finiteness of the right-hand side of (\ref{ild}) for $\mu> k' -k$ is straightforward. Hence the left-hand side of (\ref{ild}) is finite. The existence of $\fr R A\vp$  follows from Lemma \ref{L2}.
\end{proof}

Now we are ready to formulate the main inversion result for $R^*$ that follows from (\ref{o9i}) and Theorem  \ref{mt3}. Note that application of this theorem leads to some additional restrictions on the classes of functions in comparison with those in Lemma \ref{atu}.

\begin{theorem}\label{dr}
Let
\[1\le k<k'\le n-1, \quad k+k'\ge n-1, \quad n-k'-1 \le \vk< k+1,\] and suppose that $f=R^* \vp$.
 If
 \[\vp \in \tilde L^p (\agnkp), \qquad  1\le p<\frac{k'+1}{k' -k+\vk}, \]
 or \[ \vp \in \tilde C_\mu (\agnkp),  \qquad \mu> k' -k+\vk, \] then
 $\vp$ can be reconstructed from $f$ by the formula (\ref{mku}) in which  $\fr R^{-1}$ defined in accordance with Theorem \ref{mt3}.
\end{theorem}
\begin{remark} The   case $\vk=n-k'-1$ in  Theorem \ref{dr} might be of particular interest. In this case
Theorem \ref{dr} holds under the assumptions \[1\le p<\frac{k'+1}{n-k-1}, \qquad \mu> n-k-1.\]
\end{remark}
\section{Concluding Remarks}

In the present paper we suggested several straightforward inversion algorithms for the Radon transform and its dual on affine Grassmannians $\agnk$ and  $\agnkp$. In particular, the Gonzalez-Kakehi Fourier inversion  method for Schwartz functions \cite{GK03} was extended to more general  $L^p$ and continuous functions and arbitrary parity of $k$ and $k'$. To this end, we have  replaced the Fourier transform by the relevant Radon-John transform. The price for this improvement is an additional parameter $\vk$ that makes the  classes of functions not optimal. The question of how to eliminate this parameter (without using stereographic projection, as in \cite{Ru04a}), remains open and requires new  ideas.

{\bf Acknowledgements.} This work was begun when the second-named author was visiting the Department of Mathematics, Louisiana State University, in 2014-2015.  He would like to express his  gratitude to the administration of this Department for the hospitality.


\end{document}